\documentclass[10pt]{amsart}
\usepackage{amssymb}
\usepackage{dsfont}
\usepackage{amscd}
\usepackage[mathscr]{euscript} 
\usepackage[all]{xy}
\usepackage{url}
\usepackage{stmaryrd}
\usepackage{comment}
\usepackage[retainorgcmds]{IEEEtrantools}
\usepackage{hyperref}
\title[Positive characteristic fields with free operators]{Model theory of fields with free operators in positive characteristic}
\author[\"{O}. BEYARSLAN]{\"{O}zlem Beyarslan$^{\clubsuit}$}
\address{$^{\clubsuit}$Bo\v{g}azi\c{c}i \"{U}niversitesi}
\email{ozlem.beyarslan@boun.edu.tr}
\author[D. M. HOFFMANN]{Daniel Max Hoffmann$^{\diamondsuit}$}
\thanks{$^{\diamondsuit}$SDG. The first author is supported by the Narodowe Centrum Nauki grants no. 2016/21/N/ST1/01465,
and 2015/19/B/ST1/01150.}
\address{$^{\diamondsuit}$Instytut Matematyki\\
Uniwersytet Warszawski\\
Warszawa\\
Poland}
\email{daniel.max.hoffmann@gmail.com}
\urladdr{https://www.researchgate.net/profile/Daniel\_Hoffmann8}
\author[M. Kamensky]{Moshe Kamensky$^{\heartsuit}$}
\thanks{$^{\heartsuit}$This research was supported by the ISRAEL SCIENCE
FOUNDATION (grant No 1382/15)}
\address{$^{\heartsuit}$Department of Mathematics, Ben-Gurion University, Be’er-Sheva,
Israel}
\email{kamenskm@math.bgu.ac.il}
\urladdr{https://www.math.bgu.ac.il/~kamenskm}\author[P. KOWALSKI]{Piotr Kowalski$^{\spadesuit}$}
\thanks{$^{\spadesuit}$ Supported by the Narodowe Centrum Nauki grants no. 2015/19/B/ST1/01150 and 2015/19/B/ST1/01151.}
\address{$^{\spadesuit}$Instytut Matematyczny\\
Uniwersytet Wroc{\l}awski\\
Wroc{\l}aw\\
Poland}
\email{pkowa@math.uni.wroc.pl} \urladdr{http://www.math.uni.wroc.pl/\textasciitilde pkowa/ }

\DeclareMathOperator{\locus}{locus}
\DeclareMathOperator{\acl}{acl} \DeclareMathOperator{\dcl}{dcl} 
 \DeclareMathOperator{\aut}{Aut} \DeclareMathOperator{\id}{id}
 \DeclareMathOperator{\fr}{Fr} 
  
\DeclareMathOperator{\im}{im}  
\DeclareMathOperator{\ch}{char}  
 
 \DeclareMathOperator{\alg}{alg}

\DeclareMathOperator{\tp}{tp}

\DeclareMathOperator{\spec}{Spec}\DeclareMathOperator{\rat}{rat}
\DeclareMathOperator{\sep}{sep}

\DeclareMathOperator{\Alg}{Alg}

\DeclareMathOperator{\nil}{Nil}
\DeclareMathOperator{\dcf}{DCF}\DeclareMathOperator{\Mor}{Mor}\DeclareMathOperator{\scf}{SCF}\DeclareMathOperator{\red}{red}

\newtheorem{theorem}{Theorem}[section]
\newtheorem{prop}[theorem]{Proposition}
\newtheorem{lemma}[theorem]{Lemma}
\newtheorem{cor}[theorem]{Corollary}

\theoremstyle{definition}
\newtheorem{definition}[theorem]{Definition}
\newtheorem{example}[theorem]{Example}
\newtheorem{remark}[theorem]{Remark}

\newtheorem{notation}[theorem]{Notation}
\newtheorem{assumption}[theorem]{Assumption}

\begin{document}
\newcommand{\lili}{\underleftarrow{\lim }}
\newcommand{\coco}{\underrightarrow{\lim }}
\newcommand{\twoc}[3]{ {#1} \choose {{#2}|{#3}}}
\newcommand{\thrc}[4]{ {#1} \choose {{#2}|{#3}|{#4}}}
\newcommand{\Zz}{{\mathds{Z}}}
\newcommand{\Ff}{{\mathds{F}}}
\newcommand{\Cc}{{\mathds{C}}}
\newcommand{\Rr}{{\mathds{R}}}
\newcommand{\Nn}{{\mathds{N}}}
\newcommand{\Qq}{{\mathds{Q}}}
\newcommand{\Kk}{{\mathds{K}}}
\newcommand{\Pp}{{\mathds{P}}}
\newcommand{\ddd}{\mathrm{d}}
\newcommand{\Aa}{\mathds{A}}
\newcommand{\dlog}{\mathrm{ld}}
\newcommand{\ga}{\mathbb{G}_{\rm{a}}}
\newcommand{\gm}{\mathbb{G}_{\rm{m}}}
\newcommand{\gaf}{\widehat{\mathbb{G}}_{\rm{a}}}
\newcommand{\gmf}{\widehat{\mathbb{G}}_{\rm{m}}}
\newcommand{\ka}{{\bf k}}
\newcommand{\ot}{\otimes}
\newcommand{\si}{\mbox{$\sigma$}}
\newcommand{\ks}{\mbox{$({\bf k},\sigma)$}}
\newcommand{\kg}{\mbox{${\bf k}[G]$}}
\newcommand{\ksg}{\mbox{$({\bf k}[G],\sigma)$}}
\newcommand{\ksgs}{\mbox{${\bf k}[G,\sigma_G]$}}
\newcommand{\cks}{\mbox{$\mathrm{Mod}_{({A},\sigma_A)}$}}
\newcommand{\ckg}{\mbox{$\mathrm{Mod}_{{\bf k}[G]}$}}
\newcommand{\cksg}{\mbox{$\mathrm{Mod}_{({A}[G],\sigma_A)}$}}
\newcommand{\cksgs}{\mbox{$\mathrm{Mod}_{({A}[G],\sigma_G)}$}}
\newcommand{\crats}{\mbox{$\mathrm{Mod}^{\rat}_{(\mathbf{G},\sigma_{\mathbf{G}})}$}}
\newcommand{\crat}{\mbox{$\mathrm{Mod}^{\rat}_{\mathbf{G}}$}}
\newcommand{\cratinv}{\mbox{$\mathrm{Mod}^{\rat}_{\mathbb{G}}$}}
\newcommand{\ra}{\longrightarrow}
\newcommand{\bdcf}{B-\dcf}
\makeatletter
\providecommand*{\cupdot}{%
  \mathbin{%
    \mathpalette\@cupdot{}%
  }%
}
\newcommand*{\@cupdot}[2]{%
  \ooalign{%
    $\m@th#1\cup$\cr
    \sbox0{$#1\cup$}%
    \dimen@=\ht0 %
    \sbox0{$\m@th#1\cdot$}%
    \advance\dimen@ by -\ht0 %
    \dimen@=.5\dimen@
    \hidewidth\raise\dimen@\box0\hidewidth
  }%
}

\providecommand*{\bigcupdot}{%
  \mathop{%
    \vphantom{\bigcup}%
    \mathpalette\@bigcupdot{}%
  }%
}
\newcommand*{\@bigcupdot}[2]{%
  \ooalign{%
    $\m@th#1\bigcup$\cr
    \sbox0{$#1\bigcup$}%
    \dimen@=\ht0 %
    \advance\dimen@ by -\dp0 %
    \sbox0{\scalebox{2}{$\m@th#1\cdot$}}%
    \advance\dimen@ by -\ht0 %
    \dimen@=.5\dimen@
    \hidewidth\raise\dimen@\box0\hidewidth
  }%
}
\makeatother

\def\Ind#1#2{#1\setbox0=\hbox{$#1x$}\kern\wd0\hbox to 0pt{\hss$#1\mid$\hss}
\lower.9\ht0\hbox to 0pt{\hss$#1\smile$\hss}\kern\wd0}

\def\ind{\mathop{\mathpalette\Ind{}}}

\def\notind#1#2{#1\setbox0=\hbox{$#1x$}\kern\wd0
\hbox to 0pt{\mathchardef\nn=12854\hss$#1\nn$\kern1.4\wd0\hss}
\hbox to 0pt{\hss$#1\mid$\hss}\lower.9\ht0 \hbox to 0pt{\hss$#1\smile$\hss}\kern\wd0}

\def\nind{\mathop{\mathpalette\notind{}}}

\maketitle
\begin{abstract}
We give algebraic conditions about a finite commutative algebra $B$ over a field of positive characteristic, which are equivalent to the companionability of the theory of fields with ``$B$-operators'' (i.e. the operators coming from homomorphisms into tensor products with $B$). We show that, in the most interesting case of a local $B$, these model companions admit quantifier elimination in the ``smallest possible'' language and they are strictly stable. We also describe the forking relation there.
\end{abstract}

\section{Introduction}\label{secintro}
The aim of this paper is to extend the results from \cite{MS2} about model theory of ``free operators'' on fields from the case of characteristic zero to the case of arbitrary characteristics. Throughout the paper we fix a field $\ka$, and all algebras and rings considered in this paper are assumed to be commutative with unit. Let us quickly recall the set-up from \cite{MS2}. For a fixed finite $\ka$-algebra $B$ and a field extension $\ka\subseteq K$, a \emph{$B$-operator} on $K$ (called a \emph{$\mathcal{D}$-ring} (structure on $K$) in \cite{MS2}, where $\mathcal{D}(\ka)=B$) is a $\ka$-algebra homomorphism $K\to K\otimes_{\ka}B$ (for a more precise description, see Def. \ref{bopdef}). For example, a map $\partial:K\to K$ is a $\ka$-derivation if and only if the corresponding map
$$K\ni x\mapsto x+\partial(x)X+\left(X^2\right)\in K[X]/(X^2)=K\otimes_{\ka}\ka[X]/(X^2)$$
is a $B$-operator for $B=\ka[X]/(X^2)$. It is proved in \cite{MS2} that if $\ch(\ka)=0$, then a model companion of the theory of $B$-operators exists and the properties of this model companion are analyzed in \cite{MS2} as well.

However, in the case of the characteristic $p>0$ only a \emph{negative} result is provided in \cite{MS2}, i.e. \cite[Prop. 7.2]{MS2} says that if $B$ contains a nilpotent element $\alpha$ such that $\alpha^p\neq 0$, then the theory of $B$-operators has no model companion. The main result of this paper says that from the aforementioned example one can actually obtain a full \emph{characterisation} of the companionable theories of fields with $B$-operators.
More precisely, Corollary \ref{supermain} says that the theory of fields with $B$-operators has a model companion if and only if the following two conditions are satisfied:
\begin{enumerate}
\item the nilradical of $B$ coincides with the kernel of the Frobenius homomorphisms on $B$;

\item the $\ka$-algebra $B$ is either local or
$$B\cong_{\ka} \ka_1\times \ldots \times \ka_n,$$
where each $\ka\subseteq \ka_i$ is a finite separable field extension.
\end{enumerate}
We denote the model companion above by $\bdcf$. The proof of Corollary \ref{supermain} proceeds similarly as the proof of the corresponding result in \cite{K2}, where some other set-up (still including differential fields) is considered. To see that the methods from \cite{K2} work, we need to show a technical result (Corollary \ref{imm} and Corollary \ref{maint2}), which intuitively says that the fibers of prolongations with respect to $B$-operators have a stratified linear structure (note that for $B=\ka[X]/(X^2)$ these fibers are tangent spaces). We show that in the local case, the obtained theory behaves much as DCF$_p$, the theory of differentially closed fields of characteristic $p$, i.e. it is strictly stable, admits quantifier elimination in the natural language of rings with operators expanded by a function symbol for the ``inverse of Frobenius'', and the underlying field is separably closed. We also show that in the non-local case, the resulting theory can be identified with the theory ACFA$_{p,d}$ (the model companion of the theory of characteristic $p$ fields with $d$ endomorphisms) and that this theory is simple and eliminates imaginaries.

The results of this paper yield the existence of a model companion of the theory of fields with operators in some cases which were not known previously. For example, we get model companions in the following cases:
\begin{itemize}
\item several (not necessarily commuting) derivations in positive characteristic;

\item several (not necessarily commuting) non-iterative Hasse-Schmidt derivations in positive characteristic;

\item several operators combining those from the previous two items.
\end{itemize}

This paper is organized as follows. In Section \ref{secoppro}, we collect the necessary technical results which are needed in the sequel. In Section \ref{secaxioms}, we give geometric axioms (using prolongations) for theories of the shape $\bdcf$, where $B$ is local such that the nilradical of $B$ coincides with the kernel of Frobenius on $B$ (Theorem \ref{mainthm}). We also show that for finite $\ka$-algebras $B$ not satisfying the above conditions, a model companion of the theory of fields with $B$-operators either does not exists or is already known to exist (Theorem \ref{nomc} and Corollary \ref{supermain}). In Section \ref{secstable}, we show that the theories of the form $\bdcf$ discussed above have the same model-theoretic properties as the theory DCF$_{p}$. Towards this end, we show a surprisingly general result about linear independence over constants (Theorem \ref{lidi} and Corollary \ref{lindis}). In Section \ref{secend}, we discuss and speculate on other topics related with model theory of fields with free operators in positive characteristic.

We thank the Nesin Mathematics Village in \c{S}irince (Selc\"{u}k, Izmir, Turkey) for hosting the workshop in July 2016, where the work on this research was initiated; and Rahim Moosa and Thomas Scanlon, for organizing the workshop; and G\"{o}nen\c{c} Onay and David Pierce, who also participated.

We would like also to thank the referee for a careful reading of our paper and many useful suggestions.

\section{$B$-operators and $\partial$-prolongations}\label{secoppro}

Let $\ka$ be a field. Assume that $B$ is a finite $\ka$-algebra of dimension $e$, and we have a $\ka$-algebra map $\pi_B:B\to \ka$. Let $\{b_0,\ldots,b_{e-1}\}$ be a fixed $\ka$-basis of $B$ such that $\pi(b_0)=1$ and $\pi_B(b_i)=0$ for $i>0$. For convenience, we also set $d:=e-1$.
\begin{remark}\label{firstrem}
In the case when $B$ is local, we will always assume $b_0=1$.
\end{remark}

\subsection{Basic definitions}
Let us fix $\ka$-algebras $R,T$.
\begin{definition}\label{bopdef}
Let $\partial=(\partial_0,\ldots,\partial_d)$ where $\partial_0,\ldots,\partial_d:R\to T$.
\begin{enumerate}
\item If $R=T$ and $\partial_0=\id$, then we say that $\partial$ is a \emph{$B$-operator on $R$} if the corresponding map
$$R\ni r\mapsto \partial_0(r)\otimes b_0+\ldots+\partial_d(r)\otimes b_d\in R\otimes_{\ka}B$$
is a $\ka$-algebra homomorphism. We will also denote the map above by the same symbol $\partial$.

\item More generally, if the corresponding map
$$R\ni r\mapsto \partial_0(r)\otimes b_0+\ldots+\partial_d(r)\otimes b_d\in T\otimes_{\ka}B$$
is a $\ka$-algebra homomorphism, then we say that $\partial$ is a \emph{$B$-operator from $R$ to $T$} (of $\partial_0$). Note that if $\partial$ is a $B$-operator from $R$ to $T$, then $\partial_0:R\to T$ is a $\ka$-algebra homomorphism.

\item If $\partial$ is a $B$-operator from $R$ to $T$, then we define the \emph{ring of constants of $\partial$} as:
$$R^{\partial}:=\{r\in R\ |\ \partial(r)=\partial_0(r)\otimes 1_B\},$$
where $T$ is naturally considered as a $\ka$-subalgebra of $T\otimes_{\ka}B$.
\end{enumerate}
\end{definition}
It is easy to see that if $\partial$ is a $B$-operator from $R$ to $T$, then $R^{\partial}$ is subring of $R$; and, if moreover $R$ is a field, then $R^{\partial}$ is a subfield of $R$.
\begin{remark}
This is the same set-up as in \cite{MS2}, just the terminology is slightly different, which we explain below.
\begin{itemize}
\item What we call ``$B$'' here is called ``$\mathcal{D}(\ka)$'' in \cite{MS2}, and $\mathcal{D}(R)$ denotes $R\otimes_{\ka}B$ in \cite{MS2} for any $\ka$-algebra $R$.

\item As mentioned in Section \ref{secintro}, what we call a ``$B$-operator (on $K$)'' here is called a ``$\mathcal{D}$-ring (structure on $K$)'' in \cite{MS2}.
\end{itemize}
\end{remark}
\begin{example}\label{ex} We describe briefly how derivations and endomorphisms fit into this set-up.
\begin{enumerate}
\item Assume that $B=\ka^e$, $\pi$ is the projection on the first coordinate and $b_0,\ldots,b_d$ is the standard basis of $\ka^e$. Then $(\id,\partial_1,\ldots,\partial_{d})$ is a $B$-operator on $R$ if and only if $\partial_1,\ldots,\partial_{d}$ are $\ka$-algebra endomorphisms of $R$.

\item Assume that $B=\ka[X]/(X^2)$ (so $e=2$), $\pi(a+bX+(X^2))=a$ and $b_0=1_B,b_1=X+(X^2)$. Then $(\id,\partial)$ is a $B$-operator on $R$ if and only if $\partial$ is a derivation on $R$ vanishing on $\ka$.

\item If we take the $d$-th Cartesian power of the $\ka$-algebra $B$ from item $(2)$ above over $\ka$ with respect to the map $\pi$ (also from item $(2)$ above), then we get an $e$-dimensional $\ka$-algebra, which we denote $B^{\times_{\ka}d}$. The map from $B^{\times_{\ka}d}$ to $\ka$ is given by the Cartesian power structure, and we choose:
    $$b_0=(1_B,\ldots,1_B),b_1=\left(X+(X^2),0_B,\ldots,0_B\right),\ldots,b_d=\left(0_B,\ldots,0_B,X+(X^2)\right).$$
Then    $(\id,\partial_1,\ldots,\partial_{d})$ is a $B^{\times_{\ka}d}$-operator on $R$ if and only if $\partial_1,\ldots,\partial_d$ are derivations on $R$ vanishing on $\ka$.
\end{enumerate}
\end{example}
For the $\ka$-algebra $B$, we have several ideals which are of interest to us:
\begin{itemize}
\item $\nil(B)$;

\item $\ker(\pi_{B})$;

\item $\ker(\fr_{B})$.
\end{itemize}
We state below an assumption on $B$, which we will make often.
\begin{assumption}\label{ass2}
$\fr_{B}\left(\ker(\pi_{B})\right)=0$.
\end{assumption}
We describe below the meaning of Assumption \ref{ass2}.
\begin{lemma}\label{schcond}
We have the following.
\begin{enumerate}
\item The $\ka$-algebra $B$ is local if and only if
$$\left(\ker(\pi_{B})\right)^e=0.$$

\item Assumption \ref{ass2} is equivalent to saying that $B$ is local and $\nil(B)=\ker(\fr_B)$.
\end{enumerate}
\end{lemma}
\begin{proof}
To show $(1)$, assume first that $B$ is local. Then $\mathfrak{m}:=\ker(\pi_B)$ is its unique maximal ideal. Since $\dim_{\ka}B=e$, there is $i\leqslant e$ such that $\mathfrak{m}^{i+1}=\mathfrak{m}^i$. By Nakayama Lemma, we get that $\mathfrak{m}^i=0$, hence also $\mathfrak{m}^e=0$.
\\
We will show the second implication from $(1)$. If $\ker(\pi_B)^e=0$, then $\ker(\pi_B)\subseteq \nil(B)$. Take $x\in B\setminus \ker(\pi_B)$. Then $x=a1_B+b$ for some $a\in \ka^*$ and $b\in \nil(B)$. Hence $x\in B^*$, so $B$ is local with a unique maximal ideal coinciding with $\ker(\pi_B)$.

To show $(2)$, assume first that $\fr_B(\ker(\pi_B))=0$. Then $\ker(\pi_B)\subseteq \nil(B)$, hence $B$ is local with unique maximal ideal being $\ker(\pi_B)$ as above. In particular, $\nil(B)\subseteq \ker(\pi_B)$, so we also get $\nil(B)=\ker(\fr_B)$.
\\
To show the other implication in $(2)$, we notice again that if $B$ is local, then $\mathfrak{m}:=\ker(\pi_B)$ is its unique maximal ideal. By the item $(1)$, we get $\mathfrak{m}=\nil(B)$. Hence, the assumption $\nil(B)=\ker(\fr_B)$ implies that $\fr_B(\ker(\pi_B))=0$.
\end{proof}
\begin{lemma}\label{vergen}
Suppose $R,S$ are $\ka$-algebras and $\partial$ is a $B$-operator from $R$ to $S$, where $B$ is local such that the map $\partial_0$ (from Definition \ref{bopdef}) is an embedding. Then we have the following.
\begin{enumerate}
\item If $R$ and $S$ are domains, then $\partial$ extends uniquely to a $B$-operator on the fields of fractions.

\item If $R\subseteq S$ is an \'{e}tale extension of rings, then $\partial$ extends uniquely to a $B$-operator on $S$.

\item If $R\subseteq S$ is a formally smooth (or $0$-smooth) extension of rings, then $\partial$ extends to a $B$-operator on $S$.
\end{enumerate}
\end{lemma}
\begin{proof}
It follows as in the proof of \cite[Theorem 27.2]{mat}, where in the definition of  \'{e}tality one should replace ``2-nilpotent'' with ``$e$-nilpotent'' (such a change gives an equivalent definition). Similarly, for formally smooth extensions.
\end{proof}
\begin{lemma}\label{frl}
Suppose that Assumption \ref{ass2} is satisfied, i.e. $\fr_B(\ker(\pi_B))=0$. Let $R,T$ be $\ka$-algebras and $\partial:R\to T\otimes_{\ka}B$ be a $B$-operator.
Then for any $r\in R$, we have
$$\partial(r^p)=\partial_0(r^p)\otimes 1_B,$$
i.e. $\partial$ is an $R^p$-algebra map (or ``$\partial$ vanishes on $r^p$''). (Since $B$ is local, we assume that $b_0=1$, see Remark \ref{firstrem}.)
\end{lemma}
\begin{proof}
Since $\fr_{T\otimes_{\ka}B}(\ker(\pi_{T\otimes_{\ka}B}))=0$, the following diagram commutes (the top square of this diagram commutes, since $\partial$ is a homomorphism):
\begin{equation*}
 \xymatrix{  R  \ar[d]_{\partial} \ar[rrr]^{\fr_R} &  &  & R  \ar[d]^{\partial} \\
  T\otimes_{\ka}B  \ar[d]_{\pi_T} \ar[rrr]^{\fr_{T\otimes_{\ka}B}}&  &  & T\otimes_{\ka}B \\
T\ar[rrr]^{\fr_T}  &  & &  T \ar[u]^{\iota_T},}
\end{equation*}
which gives the result, since $\pi_T\circ \partial=\partial_0$.
\\
Equivalently, one can see the desired equality in the following way:
$$\partial(r^p)=(\partial(r))^p=\left(\sum_i\partial_i(r)\otimes b_i\right)^p=\partial_0(r^p)\otimes 1_B,$$
where the last equality holds, since $\ch(\ka)=p$, $\partial_0$ is a homomorphism, and $b_i^p=0$ for $i>0$.
\end{proof}
Let $B_{\red}:=B/\nil(B)$ be the associated reduced ring with the induced $\ka$-algebra map $\pi_{\red}:B_{\red}\to \ka$. Note that, since $B$ is $\ka$-algebra, the reduction map $B\to B_{\red}$ has a canonical section $s_B:B_{\red}\to B$, so we can view $B_{\red}$ as a sub-algebra of $B$.  If
$\partial:R\to{}T\otimes_{\ka}B$ is a $B$-operator, then we write $\partial_{\red}:R\to{}T\otimes_{\ka}B_{\red}$ for the corresponding $B_{\red}$-operator, which is the reduction of $\partial$, and we also define the following sub-algebra of $R$:
$$R_r:=\{r\in R\ |\ \partial(r)=\partial_{\red}(r)\}.$$
It is easy to see that in the local case one gets $B_{\red}=\ka$ and $\partial_{\red}=\partial_0$, therefore $R_r=R^{\partial}$.

We show below a partial converse to Lemma~\ref{frl}: a $B_{\red}$-operator can be
lifted to a $B$-operator, as long as it satisfies the necessary condition
of Lemma \ref{frl}.

\begin{lemma}\label{pextend}
Suppose that
\begin{itemize}
  \item $\ka\subseteq K\subseteq M$ is a tower of fields and $T$ is a $\ka$-algebra;
  \item $\partial^K:K\to{}T\otimes_{\ka}B$ is a $B$-operator such that $M^p\subseteq K_r$;
  \item  $\partial^M_r:M\to{}T\otimes_{\ka}B_{\red}$ is a $B_{\red}$-operator such that
  $$\left(\partial^M_r\right)|_K=\left(\partial^K\right)_{\red}.$$
\end{itemize}
Then there exists a lifting $\partial^M:M\to{}T\otimes_{\ka}B$ of $\partial_r^M$ restricting to $\partial^K$.
\end{lemma}
\begin{proof}
  Since $M^p\subseteq{}K$, $M$ is a purely inseparable extension of
  $K$, and we may assume by induction that it is of the form
  $M=K(t^{1/p})$, where $\partial^K(t)=(\partial^K)_{\red}(t)$ (since $t\in R_r$). Since we have:
  $$\partial^M_r(t^{1/p})^p=\partial^M_r(t)=(\partial^K)_{\red}(t)=\partial^K(t),$$
we may set
$$\partial^M(t^{1/p}):=(\id_T\otimes s_B)\left(\partial^M_r(t^{1/p})\right),\ \ \ \ \ \ \ \ \ \partial^M|_K:=\partial^K;$$
and such $\partial^M$ is well-defined and is a lifting of $\partial_r^M$ restricting to $\partial^K$.
\end{proof}

\begin{remark}
  The same proof works when $B_{\red}$ is replaced by any quotient of $B$,
  as long as it has a section, so that the analogue of $R_r$ can be defined.
  For instance, after replacing $B_{\red}$ with $\ka$, the role of $R_r$ is played by the ring of constants $R^\partial$ (see Definition \ref{bopdef}(3)).
\end{remark}
Let us fix a field extension $\ka\subseteq K$ and a $B$-operator $\partial:K\to B\otimes_{\ka}K$ on $K$. We tacitly assume that all the fields considered are subfields of a (big) algebraically closed field $\Omega$.

For an affine $K$-scheme $V$, we want to define its prolongation $\tau^{\partial}(V)$. The defining property of $\tau^{\partial}(V)$ is that for any $K$-algebra $R$, we should have a natural bijection between the following sets of rational points over $K$:
$$\tau^{\partial}(V)(R)\ \ \longleftrightarrow\ \  V\left(B\otimes_{\ka}R\right),$$
where $B\otimes_{\ka}R$ has the $K$-algebra structure given by the composition of ${\partial}$ with the map $B\otimes_{\ka}K\to B\otimes_{\ka}R$. Since we are interested only in affine varieties, we are in fact looking for the \emph{left-adjoint} functor to the following functor:
$$B^{\partial}(R):\Alg_K \to \Alg_K,\ \ \ B^{\partial}(R):=B\otimes_{\ka}R,$$
where the $K$-algebra structure on $B^{\partial}(R)$ is described above. It is easy to see that for any $K$-algebra map $f:R\to R'$ the induced map $B(f):B^{\partial}(R)\to B^{\partial}(R')$ is also a $K$-algebra map, so we get a functor indeed. We are looking for a left-adjoint functor to the functor $B^{\partial}$, i.e. a functor
$$\tau^{\partial}:\Alg_K \to \Alg_K$$
such that for any $K$-algebras $R,S$, there is a natural bijection:
$$\Mor_{\Alg_K}\left(R,B^{\partial}(S)\right)\longleftrightarrow \Mor_{\Alg_K}\left(\tau^{\partial}(R),S\right).$$
This functor is described at the end of Section 3 of \cite{MS2}.
\begin{remark}\label{empty}
We have to accept that the $0$-ring is a $K$-algebra, since it may appear as $\tau^{\partial}(R)$ for some $R$. For example, if $\partial$ is a derivation and $a\in K\setminus K^{\partial}$, then we have
$$\tau^{\partial}\left(K(a^{1/p})\right)=\{0\}.$$
\end{remark}
We describe now several natural maps. Consider the adjointness bijection:
$$\Mor_{\Alg_K}\left(\tau^{\partial}(R),\tau^{\partial}(R)\right)\to \Mor_{\Alg_K}\left(R,B^{\partial}(\tau^{\partial}(R))\right).$$
\begin{remark}\label{natrem}
\begin{enumerate}
\item The image of the identity map by the bijection above is a natural $B$-operator extending $\partial$
$$\partial_R:R\to \tau^{\partial}(R)\otimes_{\ka}B.$$

\item We define the map $\pi_{\partial}^R:R\to \tau^{\partial}(R)$ as the composition of the following maps:
\begin{equation*}
 \xymatrix{R\ar[rr]^{\partial_R\ \ \ \ \ \ \ } &  & B^{\partial}(\tau^{\partial}(R)) \ar[rr]^{\pi_{\tau^{\partial}(R)}} &  & \tau^{\partial}(R).}
\end{equation*}
\end{enumerate}
\end{remark}
If $V=\spec(R)$ is an affine $K$-scheme, then we define its \emph{$\partial$-prolongation} as
$$\tau^{\partial}(V):=\spec\left(\tau^{\partial}(R)\right),$$
which is also a $K$-scheme. By Remark \ref{empty}, it may happen that $\tau^{\partial}(V)$ is the ``empty scheme'' (for a non-empty $V$). The map $\pi_{\partial}^R$ gives us the following natural morphism:
$$\pi_{\partial}^V:\tau^{\partial}(V)\to V.$$
For any $K$-algebra $R$ and any $a\in V(R)$, we denote by $\tau^{\partial}_a(V)$ the scheme over $R$, which is the fiber of the morphism $\pi_{\partial}^V$ over $a$.
\begin{remark}\label{partialv}
There is a natural (in $V$) map (not a morphism!):
$$\partial_V:V(K)\to \tau^{\partial}V(K)=V\left(B^{\partial}(K)\right)$$
given by $V(\partial)$, where we consider $V$ as a functor (of rational points) from the category of $K$-algebras to the category of sets and we apply this functor to the $K$-algebra homomorphism $\partial:K\to B^{\partial}(K)$.
\end{remark}
\noindent
By Remark \ref{natrem}(1), we immediately get the following.
\begin{lemma}\label{easylemma}
Suppose $V$ and $W$ are $K$-varieties and $W\subseteq \tau^{\partial}(V)$. Then we get a natural $B$-operator
$$\partial^W_V:K[V]\to K[W]\otimes_{\ka}B,$$
which  extends $\partial$ and is obtained as the composition of $\partial_{K[V]}$ with the $K$-algebra morphism $K[\tau^{\partial}(V)]\to K[W]$ induced by the inclusion morphism $W\to \tau^{\partial}(V)$.
\end{lemma}

\subsection{Rational points of $\partial$-prolongations}
We state here our main technical result. It may look rather technical indeed, but the reader should have in mind the immediate application, which is Corollary \ref{imm}. This corollary in the case of derivations reduces to the fact that tangent spaces are vector spaces, so if they are defined over a field $L$, they must have $L$-rational points. In general, Corollary \ref{imm} may be understood as saying that if $B$ satisfies Assumption \ref{ass2}, then the fibers of the prolongations with respect to a $B$-operator posses a stratified linear structure (and have $L$-rational points as above).
\begin{prop}\label{maint}
Suppose that $\fr_B(\ker(\pi_B))=0$ and:
\begin{itemize}
\item $\ka\subseteq K\subseteq L$ and $\ka\subseteq K\subseteq M$ are towers of fields;

\item $\partial:K\to L\otimes_{\ka}B$ is a $B$-operator;

\item $c:M\to \Omega\otimes_{\ka}B$ is a $B$-operator and $b:M\to L$ is a $K$-algebra map such that the following diagram is commutative:
\begin{equation*}
 \xymatrix{  K  \ar[d]_{\subseteq} \ar[rr]^{\partial}  &  &  L\otimes_{\ka}B  \ar[d]^{\subseteq} \\
M  \ar[d]_{b} \ar[rr]^{c}  &  &  \Omega \otimes_{\ka}B  \ar[d]^{\id_{\Omega}\otimes \pi}\\
L \ar[rr]^{\subseteq} &  &  \Omega.}
\end{equation*}
\end{itemize}
Then there is $B$-operator $c':M\to L\otimes_{\ka}B$ such that the following diagram is commutative:
\begin{equation*}
 \xymatrix{  K  \ar[d]_{\subseteq} \ar[rr]^{\partial}  &  &  L\otimes_{\ka}B \ar[d]^{=}\\
M  \ar[d]_{b}  \ar[rr]^{c'}  &  &  L\otimes_{\ka}B  \ar[d]^{\id_{L}\otimes \pi}\\
L \ar[rr]^{=} &  &  L.}
\end{equation*}
\end{prop}
\begin{proof}
Our aim is to expand the map $b$ to a $B$-operator
$$c':M\to L\otimes_{\ka}B$$
extending $\partial$ on $K$. There is a subfield $M_0\subseteq M$ such that the field extension $K\subseteq M_0$ is separable and the extension $M_0\subseteq M$ is purely inseparable. By Lemma \ref{vergen}(3) (since a separable field extension is formally smooth, see \cite[Theorem 26.9]{mat}), without loss of generality we can assume that $K=M_0$. Then, by an easy induction, we can assume that $M=K(t^{1/p})$ for some $t\in K$. By Lemma \ref{frl}, we get that for each $i>0$ we have $c_i(t)=0$. Therefore $t\in K^{\partial}=K_r$ (in the notation introduced before Lemma \ref{pextend}), so $M^p\subseteq K_r$. Applying Lemma \ref{pextend} for $\partial^K:=\partial$, $\partial^M_r:=b$ and $T:=L$, we get (since $B_{\red}=\ka$ in our case here) the required $B$-operator $c'$.
\end{proof}
From now on, by a \emph{$K$-variety} (or a \emph{variety over $K$}) we mean a $K$-irreducible and $K$-reduced affine scheme over $K$ of finite type. Hence, a  $K$-variety for us is basically the same as a prime ideal in a polynomial ring over $K$ in finitely many variables.
\begin{cor}\label{imm}
Suppose that $\fr_B(\ker(\pi_B))=0$ and:
\begin{itemize}
\item $\ka\subseteq K\subseteq L$ is a tower of fields;

\item $\partial$ is a $B$-operator on $K$;

\item $W$ is a variety over $K$;

\item there is $c\in \tau^{\partial}W(\Omega)$ such that
$$b:=\pi_{\partial}(c)\in W(L)$$
is a generic point of $W$ over $K$.
\end{itemize}
Then $\tau^{\partial}_bW(L)\neq \emptyset$.
\end{cor}
\begin{proof}
We consider the rational point $b\in W(L)$ as a $K$-algebra homomorphism and $c\in \tau^{\partial}W(\Omega)$ as a $B$-operator expanding $b$, so we get a commutative diagram (almost) as in the assumptions of Proposition \ref{maint}:
\begin{equation*}
 \xymatrix{  K  \ar[d]_{\subseteq} \ar[rr]^{\partial}  &  &  K\otimes_{\ka}B  \ar[d]^{\subseteq} \\
K[W]  \ar[d]_{b} \ar[rr]^{c}  &  &  \Omega \otimes_{\ka}B  \ar[d]^{\id_{\Omega}\otimes \pi}\\
L \ar[rr]^{\subseteq} &  &  \Omega.}
\end{equation*}
Since $b$ is a generic point of $W$ over $K$, the corresponding map $b:K[W]\to L$ is one-to-one. Let $M:=K(W)$, which is the field of fractions of $K[W]$. By Lemma \ref{vergen}(1), the $B$-operator $c$ extends to $M$ and now we are exactly in the situation from Proposition \ref{maint}. By Proposition \ref{maint}, we get a $B$-operator $c'$ which is an element of $\tau^{\partial}_bW(L)$, hence this fiber is non-empty.
\end{proof}
\begin{example}
Corollary \ref{imm} (hence also Proposition \ref{maint}) is not true without making Assumption \ref{ass2}, which an easy example below shows. Let us take $p=2$, $B:=\Ff_2[X]/(X^3)$ and $K:=\Ff_2(x,y)$, where $x$ and $y$ are algebraically independent over $\Ff_2$. We put a $B$-structure on $K$ by declaring that $\partial_1(x)=0$ and $\partial_2(x)=y$. We also take:
$$ L=K\left(x^{1/2}\right),\ W=\spec\left(K[X]/\left(X^2-x\right)\right),\ c=\left(y^{1/2},x^{1/2}\right),\ b=x^{1/2}.$$
Then the assumptions of Corollary \ref{imm} are satisfied (obviously, Assumption \ref{ass2} is not), but $\tau_bW(L)$ is empty (even $\tau_bW(L^{\sep})$ is empty).
\end{example}
We proceed now to show Corollary \ref{maint2}, which a variant of Corollary \ref{imm} and still follows directly from Proposition \ref{maint}. This variant will be crucial for the axiomatization of existentially closed fields with $B$-operators (Section \ref{secaxioms}). Assume that $V,W$ are $K$-varieties and $W\subseteq \tau^{\partial}(V)$. Let $\iota:W\to \tau^{\partial}(V)$ denote the inclusion morphism and
$$\alpha:=\pi^V_{\partial}\circ \iota:W\to V.$$
Consider the following (non-commutative!) diagram:
\begin{equation*}
 \xymatrix{  &  &  \tau^{\partial}(W)  \ar[lld]_{\pi^W_{\partial}} \ar[rrd]^{\tau^{\partial}(\alpha)} &  &   \\
W\ar[rrrr]^{\iota}  &  & &  &  \tau^{\partial}(V) .}
\end{equation*}
Using this diagram, we define the following $K$-subvariety of $\tau^{\partial}(W)$:
\begin{IEEEeqnarray*}{rCl}
E &:=& \mathrm{Equalizer}\left(\tau^{\partial}(\alpha),\iota \circ \pi^W_{\partial}\right)\\
  &=& \left\{a\in \tau^{\partial}(W)\ |\ \tau^{\partial}(\alpha)(a)=\iota \circ \pi^W_{\partial}(a)\right\}.
\end{IEEEeqnarray*}
\begin{example}\label{exe}
Assume that:
$$e=3,\ \ V=\Aa^1,\ \ W=\tau^{\partial}(V)=\Aa^3.$$
Then we have $\tau^{\partial}(W)=\Aa^9$. Let us name typical rational points as follows:
$$x\in V(K),\ (x,y,z)\in W(K),\ \left(x,x',x'',y,y',y'',z,z',z''\right)\in \tau^{\partial}(W)(K).$$
Using the above notation, we obtain:
\begin{IEEEeqnarray*}{rCl}
\tau^{\partial}(\alpha)\left(x,x',x'',y,y',y'',z,z',z''\right) &=& \left(x,x',x''\right);\\
\iota\left(\pi_W\left(x,x',x'',y,y',y'',z,z',z''\right)\right) &=& (x,y,z).
\end{IEEEeqnarray*}
Hence $E\subseteq \Aa^9$ is given by the equations $x'=y,x''=z$.
\end{example}
We have the following commutative diagram, where the morphism $\pi_E$ is defined as the restriction of the morphism $\pi^W_{\partial}:\tau^{\partial}(W)\to W$ to $E$ (as the diagram suggests)
\begin{equation*}
 \xymatrix{  E  \ar[dd]_{\subseteq} \ar[rr]^{\subseteq} \ar[rd]^{\pi_E} &  &  \tau^{\partial}(W)  \ar[dd]^{\tau^{\partial}(\alpha)} \\
 & W\ar[rd]^{\iota} &  \\
\tau^{\partial}(W)\ar[ru]^{\pi^W_{\partial}}\ar[rr]^{\iota\circ \pi^W_{\partial}} &  &  \tau^{\partial}(V).}
\end{equation*}
We can state now the aforementioned version of Corollary \ref{imm}.
\begin{cor}\label{maint2}
Suppose that $\fr_B(\ker(\pi_B))=0$ and:
\begin{itemize}
\item $\ka\subseteq K\subseteq L$ is a tower of fields;

\item $\partial$ is a $B$-operator on $K$;

\item $V$ and $W$ are varieties over $K$;

\item $W\subseteq \tau^{\partial}(V)$;

\item the projection morphism $W\to V$ is dominant;

\item there is $c\in E(\Omega)$ such that
$$b:=\pi_E(c)\in W(L)$$
is a generic point in $W$ over $K$.
\end{itemize}
Then $\pi_E^{-1}(b)(L)\neq \emptyset$.
\end{cor}
\begin{proof}
Since $W\subseteq \tau^{\partial}(V)$, by Lemma \ref{easylemma}, we get a $B$-operator $\partial_V^W$ extending $\partial$ from $K[V]$ to $K[W]$. Since the projection morphism $W\to V$ is dominant, by Lemma \ref{vergen}(1), $\partial_V^W$ extends to a $B$-operator (denoted by the same symbol) $\partial_V^W$ from $K(V)$ to $K(W)$. Since $b\in W(L)$ is generic, it gives a $K$-algebra map $b:K(W)\to L$. Therefore, we define the $B$-operator $\partial_V$ from $K(V)$ to $L$ as the composition of the following maps:
\begin{equation*}
 \xymatrix{K(V)\ar[rr]^{\partial_V^W\ \ \ \ } & & K(W)\otimes B \ar[rr]^{\ \ b\otimes \id_B} & & L\otimes B.}
\end{equation*}
As in the proof of Corollary \ref{imm}, the assumption $c\in \tau^{\partial}(W)(\Omega)$ means that $c$ is a $B$-operator from $K(W)$ to $\Omega$. The extra assumption that
$c\in E(\Omega)$ says exactly that $c$ extends the $B$-operator $\partial_V$ defined above. Therefore, we can apply Proposition \ref{maint} for $K(V)$ playing the role of $K$, $\partial_V$ playing the role of $\partial$ and (as in the proof of Corollary \ref{imm}) $M:=K(W)$. By Proposition \ref{maint}, we get a $B$-operator $c'$ which is an element of $\pi_E^{-1}(b)(L)$.
\end{proof}

\subsection{Tensor products}\label{secten}
In this subsection, we do not put any restrictions on the finite $\ka$-algebra $B$. It is well known that if we have two differential ring extensions of a given differential ring, then their tensor product (over this given ring) has a natural unique differential structure extending the differential structure on each of its factors (similarly in the difference case). We will generalize the above results to arbitrary $B$-operators.

Assume that $(K,\partial)$ is a field with a $B$-operator and that $(R,\partial_R),(S,\partial_S)$ are $K$-algebras with $B$-operators extending $\partial$. Consider the following commutative diagram of $\ka$-algebras:
\begin{equation*}
 \xymatrix{   &  & \left(R\otimes_K S\right)\otimes_{\ka}B &  &   \\
   &  &   &  &   \\
 R\otimes_{\ka}B\ar[rruu]^{t_R}  &  & R\otimes_K S \ar[uu]^{\left(\widetilde{\partial}_R,\widetilde{\partial}_S\right)} &  & S\otimes_{\ka}B\ar[lluu]_{t_S}  \\
  &  &  K\otimes_{\ka}B \ar[rru]^{\iota_S\otimes\id}\ar[llu]_{\iota_R\otimes\id} &  &  \\
R\ar[uu]^{\partial_R}\ar[rruu]^{}  &  &    &  &  S\ar[uu]_{\partial_S}\ar[lluu]^{} \\
  &  &  K, \ar[llu]_{\iota_R}\ar[rru]^{\iota_S} \ar[uu]^{\partial}&  & }
\end{equation*}
where:
\begin{itemize}
  \item $t_R$ is the tensor product over $\ka$ of the natural embedding $R\to R\otimes_KS$ and $\id_B$ (similarly for $t_S$);
  \item $\widetilde{\partial}_R:=t_R\circ \partial_R,\ \widetilde{\partial}_S:=t_S\circ \partial_S$;
  \item the map $(\widetilde{\partial}_R,\widetilde{\partial}_S)$ is given by the universal property (being the coproduct) of the tensor product over $K$, that is:
      $$(\widetilde{\partial}_R,\widetilde{\partial}_S)(a\otimes b)=\widetilde{\partial}_R(a)\cdot \widetilde{\partial}_S(b),$$
      where the product $\cdot$ occurs in the $\ka$-algebra $(R\otimes_KS)\otimes_{\ka}B$.
\end{itemize}
We can define now the tensor product (over $K$) of the $B$-operators $\partial_R,\partial_S$ as the map $(\widetilde{\partial}_R,\widetilde{\partial}_S)$, and we have proved the following.
\begin{prop}\label{tensor}
If $(K,\partial)$ is a field with a $B$-operator and $(R,\partial_R),(S,\partial_S)$ are $K$-algebras with $\partial$-operators, then there is a unique  $B$-operator on $R\otimes_KS$ such that the natural maps $R\to R\otimes_KS,S\to R\otimes_KS$ preserve the corresponding $B$-operators.
\end{prop}

\section{Axioms}\label{secaxioms}
We still use the fixed $\ka,B,\pi_B,e$ from the beginning of Section \ref{secoppro}. We consider the language of rings extended by $d=e-1$ extra unary symbols and extra constant symbols for the elements of $\ka$. Then we get a first-order theory of rings/fields with $B$-operators in the language defined above.  We use the name \emph{$B$-field}, when we talk about a field with a $B$-operator. Similarly for \emph{$B$-extensions}, \emph{$B$-isomorphisms}, etc.

Assume now that $\ch(\ka)=p>0$. Our aim in this section is to prove the converse implication to Theorem \ref{nomc} below, and this theorem (after a bit of work) can be extracted from \cite{MS2} and \cite{Pierce2}. We need a lemma first.
\begin{lemma}\label{gencon}
Assume that $\ker(\fr_B)=\nil(B)$ and:
$$B\cong_{\ka} B_1\times \ldots \times B_l$$
such that $B_1,\ldots,B_l$ are local, the residue field of each $B_i$ coincides with $\ka$, and $B_j\neq \ka$ for some $j\in \{1,\ldots,l\}$. Assume also that $\partial$ is a $B$-operator on $K$.
\begin{enumerate}
  \item Then $\partial$ is of the form:
$$\partial=\left(\sigma_1,\ldots,\sigma_l,\partial_1,\ldots,\partial_N\right),$$
where $\sigma_1,\ldots,\sigma_l$ are endomorphisms of $K$ (one of them is the identity map), $N>0$ and $\partial_1,\ldots,\partial_N$ vanish on $K^p$.
  \item Suppose that $a\in K\setminus K^p$ is such that for all $i_1,\ldots,i_n\in \{1,\ldots,l\}$ (they may repeat) and all $j\in \{1,\ldots,N\}$ we have:
$$\left(\partial_j\circ \sigma_{i_1}\circ \ldots \circ \sigma_{i_n}\right)(a)=0.$$
Then $\partial$ extends to a $B$-operator on a field $L$ such that:
$$K\left(a^{1/p}\right)\subseteq L\subseteq K^{1/p}.$$
\end{enumerate}
\end{lemma}
\begin{proof}
The first item is clear by Lemma \ref{frl}.

For the proof of the second item, let us define:
$$L:=K\left(\left(\sigma_{i_1}\circ \ldots \circ \sigma_{i_n}\right)(a)^{1/p}\ |\ i_1,\ldots,\i_n\in \{1,\ldots,l\},n\in \Nn\right).$$
The endomorphisms $\sigma_1,\ldots,\sigma_l$ clearly extend to (unique) endomorphisms $\sigma_1',\ldots,\sigma_l'$ of $L$, i.e. $L$ is a
$B_{\red}$-field in the terminology introduced before Lemma \ref{pextend}.  By our assumption on
$a$ (and Lemma \ref{frl}), we have
$$L^p\subseteq K_0:=\bigcap_{j=1}^N\ker(\partial_j)=K_r$$
(again in the terminology introduced before Lemma \ref{pextend}). Thus, the assumptions of Lemma \ref{pextend} are satisfied, and we are done.
\end{proof}
\begin{theorem}\label{nomc}
Suppose that the theory of fields with $B$-operators has a model companion. Then the following two conditions are satisfied:
\begin{enumerate}
\item $\nil(B)=\ker(\fr_B)$;

\item $B$ is either local or
$$B\cong_{\ka} \ka_1\times \ldots \times \ka_n,$$
where each $\ka\subseteq \ka_i$ is a finite separable field extension.
\end{enumerate}
\end{theorem}
\begin{proof} The item $(1)$ is exactly \cite[Proposition 7.2]{MS2}.

To show the item $(2)$, we note first that by the Structure Theorem for Artinian Rings (see \cite[Theorem 8.7]{atmac}), we have the following isomorphism:
$$B\cong_{\ka} B_1\times \ldots \times B_l,$$
where $B_1,\ldots,B_l$ are local finite $\ka$-algebras (the statement of \cite[Theorem 8.7]{atmac} is about Artin \emph{rings} only, but it is easy to see that the same proof gives a $\ka$-algebra isomorphism). Assume now that $\partial$ is a $B$-operator on $K$ and that $B$ satisfies the condition in the item $(1)$ and does not satisfy the condition in the item $(2)$.

Assume first that we are in the situation from Lemma \ref{gencon} (i.e. the residue field of each $B_i$ coincides with $\ka$ and some $B_j$ is not a field). Then $\partial$ is of the form:
$$\partial=\left(\sigma_1,\ldots,\sigma_l,\partial_1,\ldots,\partial_N\right),$$
where $\sigma_1,\ldots,\sigma_l$ are endomorphisms of $K$ (one of them is the identity map, say $\sigma_1=\id$), $l>1,N>0$ and  $\partial_1,\ldots,\partial_N$ vanish on $K^p$.
We can finish the proof exactly as in \cite[Theorem 5.2]{Pierce2}, which we describe below. For $m\in \Nn$, let $\mathbf{K}_m$ be the field $\Ff_p(X_0,X_1,\ldots)$ with a $B$-operator such that $\sigma_2,\ldots,\sigma_l$ move $X_i$ to $X_{i+1}$ for all $i\in \Nn$ (so this does not depend on $m$)  and for all $j\in \{1,\ldots,N\}$, we have:
$$\partial_j(X_0)=0,\ \ldots\ ,\ \partial_j(X_{m-1})=0,\ \partial_j(X_{m})=1,\ \partial_j(X_{m+1})=1,\ \ldots \ .$$
For each $m\in \Nn$, we extend $\mathbf{K}_m$ to an existentially closed field with a $B$-operator $\mathbf{L}_m=(L_m,\partial^{\mathbf{L}_m})$. Let $\mathbf{N}=(N,\partial^{\mathbf{N}})$ be the ultraproduct (with respect to a non-principal ultrafilter) of $\mathbf{L}_m$'s. For each $m\in \Nn$, we have:
$$\partial_1^{\mathbf{L}_m}\left(\left(\sigma_2^{\mathbf{L}_m}\right)^m(X_0)\right)=\partial_1^{\mathbf{L}_m}\left(X_m\right)=1\neq 0,$$
hence $X_0\notin L_m^p$. By {\L}o\'{s}'s theorem, $X_0\notin N^p$.  Let us fix now $i_1,\ldots,i_n\in \{1,\ldots,l\}$  and $j\in \{1,\ldots,N\}$. Then for all $m>n$, we have:
$$\left(\partial_j^{\mathbf{L}_m}\circ \sigma_{i_1}^{\mathbf{L}_m}\circ \ldots \circ \sigma_{i_n}^{\mathbf{L}_m}\right)(X_0)=0.$$
By {\L}o\'{s}'s theorem, we get that for all $i_1,\ldots,i_n\in \{1,\ldots,l\}$  and all $j\in \{1,\ldots,N\}$, we have:
$$\left(\partial_j^{\mathbf{N}}\circ \sigma_{i_1}^{\mathbf{N}}\circ \ldots \circ \sigma_{i_n}^{\mathbf{N}}\right)(X_0)=0.$$
Using Lemma \ref{gencon}(2) (for $a:=X_0$), we see that $\mathbf{N}$ is not an existentially closed field with a $B$-operator, which is a contradiction.

Consider now the general case and let $\ka\subseteq K$ be a field extension. Let us denote
$$B_K:=B\otimes_{\ka}K.$$
We have the notions of $B_K$-fields, $B_K$-field extensions etc. It is easy to see that if $(L,\partial)$ is a $B_K$-field, then the notion of a $B$-field extension of $(L,\partial)$ coincides the notion of a $B_K$-field extension of $(L,\partial)$. Therefore, if the theory of $B$-fields has a model companion, then the theory of $B_K$-fields has a model companion as well. Let us take $K:=\ka^{\alg}$. Then $B_K$ decomposes as the product of local $K$-algebras such that all the residue fields coincide with $K$. To apply the case which was already proved above (and, therefore, to finish the proof), it is enough to notice that not all the local $K$-algebras in this decomposition are fields. If this is not the case (that is, if $B_K$ is a product of fields), then $B_K$ is reduced. However, for each $i\in \{1,\ldots, l\}$ both $B_i$ and $(B_i)_K$ embed into $B_K$ (as $\ka$-algebras). Therefore, if some $B_i$ is not a field, then $B_i$ is not reduced (by an easy argument, e.g. as in the first paragraph of the proof of Lemma \ref{schcond}), and then $B_K$ is also not reduced. If some $B_i$ is a field extension of $\ka$ which is not separable, then, by \cite[(27.C) Lemma 1]{mat0}, $(B_i)_K$ is not reduced, and then $B_K$ is not reduced either.
\end{proof}
\begin{remark}
By Theorem \ref{nomc} and Lemma \ref{schcond}, we obtain that (while trying to show the existence of a model companion) one can safely make Assumption \ref{ass2} from Section \ref{secoppro} (see also Corollary \ref{supermain}). Therefore, we make this assumption now (except for Corollary \ref{supermain}), that is we assume that:
$$\fr_{B}\left(\ker(\pi_{B})\right)=0,$$
even when this assumptions is not necessary in some of the results we will show below. Note that, by Lemma \ref{schcond}, this assumption is equivalent to $B$ being local and such that
$$\ker(\fr_B)=\nil(B).$$
\end{remark}
Let us fix a field extension $\ka\subseteq K$ and a $B$-operator $\partial$ on $K$. For $K$-algebras $R$ and $T$, we will consider $B$-operators from $R$ to $T$ extending $\partial$. We will call them \emph{$\partial$-operators} and sometimes write them in the form
$$\bar{D}=(f,D):R\to T,$$ where $f=\bar{D}_0$ is a $K$-algebra map from $R$ to $T$. It will be often the case that $f$ is the inclusion map. We also fix a field extension $K\subseteq \Omega$, where $\Omega$ is a big saturated algebraically closed field.
\begin{notation}\label{not}
We will use the following terminology originating from the differential case considered in \cite{Lando} (which was modeled on the difference case from \cite{cohn}). Let $n>0$ and $a\in \Omega^n$, $a',b\in \Omega^{nd}$, $b'\in \Omega^{nd^2}$.
\begin{itemize}
\item If $\bar{D}:K[a]\to K[a,a']$ is a $\partial$-operator such that $\bar{D}(a)=(a,a')$ (i.e. $f$ above is the inclusion map and $D(a)=a'$), then we call $\bar{D}$ a \emph{$B$-kernel}.
\item If $\bar{D}:K[a]\to K[a,a']$ is a $B$-kernel and $\bar{D}':K[a,a']\to K[a,a',b,b']$ is a $B$-kernel such that $\bar{D}'|_{K[a]}=\bar{D}$ (in particular, $a'=b$), then we call $\bar{D}'$ a \emph{$B$-prolongation} of $\bar{D}$.
\item If $\bar{D}:K[a]\to K[a,a']$ is a $B$-kernel, $L$ is a field extension of $K(a,a')$ and $\partial'$ is a $B$-operator on $L$ (in particular, we have $\partial'(L)\subseteq L$) extending $\bar{D}$, then we call $(L,\partial')$ a \emph{$B$-regular realization} of $\bar{D}$.
\end{itemize}
\end{notation}
For the next two results (Lemma \ref{easylemma2} and Proposition \ref{kerprol}), we fix the following data:
\begin{itemize}
\item $n>0$, $a\in \Omega^n$ and $a'\in \Omega^{nd}$;

\item $V=\locus_K(a)$;

\item $W=\locus_K(a,a')$.
\end{itemize}
We state below a very general fact, which follows immediately from Lemma \ref{easylemma}.
\begin{lemma}\label{easylemma2}
The following are equivalent.
\begin{enumerate}
\item The extension $K\subseteq K[a,a']$ has a $B$-kernel structure such that $D(a)=a'$.

\item $W\subseteq \tau^{\partial}(V)$.
\end{enumerate}
\end{lemma}
We assume now that $K[a]\subseteq K[a,D(a)]$ is a $B$-kernel, so $a'=D(a)$ and
$$V=\locus_K(a),\ \ \ W=\locus_K(a,D(a)).$$
We can prove now the main criterion about prolonging $B$-kernels to $B$-regular realizations. It is analogous to and plays the same role as Kernel-Prolongation Lemmas from \cite{BK} (e.g. Lemma 2.13, Proposition 3.7 or Proposition 3.17 in \cite{BK}). We advice the reader to recall the definition of the $K$-subvariety $E\subseteq \tau^{\partial}(W)$ and the morphism $\pi_E$, which appear before Corollary \ref{maint2}.
\begin{prop}[Kernel-Prolongation Lemma]\label{kerprol}
The following are equivalent.
\begin{enumerate}
\item The morphism $\pi_E:E\to W$ is dominant.

\item The $B$-kernel $\bar{D}:K[a]\to K[a,D(a)]$ has a $B$-regular realization on the field $K(a,D(a))$.

\item The $B$-kernel $\bar{D}:K[a]\to K[a,D(a)]$ has a $B$-regular realization.

\item The $B$-kernel $\bar{D}:K[a]\to K[a,D(a)]$ has a $B$-prolongation.
\end{enumerate}
\end{prop}
\begin{proof}
$(1)\Rightarrow (2)$. Since the projection morphism $\pi_E:E\to W$ is dominant and $(a,D(a))$ is a generic point of $W$, we get that $(a,D(a))$ is of the form $\pi_E(c)$ for some $c\in E(\Omega)$, hence the assumptions of Corollary \ref{maint2} (for $b:=(a,D(a))$ and $L:=K(b)$) are satisfied. We note that by the definition of the morphism $\pi_E$, for any $K(a,D(a))$-algebra $T$, we have:
\begin{equation}
\pi_E^{-1}(a,D(a))(T)=\left\{t\in T^{nd^2}\ |\ \left(a,D(a),D(a),t\right)\in \tau^{\partial}W(T)\right\}.\tag{$*$}
\end{equation}
By Corollary \ref{maint2}, we have:
\begin{equation}
\pi_E^{-1}(a,D(a))(K(a,D(a)))\neq \emptyset.\tag{$**$}
\end{equation}
Using $(*)$ and $(**)$ (for $T=K(a,D(a))$), we can find a tuple $b'$ such that:
$$(a,D(a),D(a),b')\in \tau^{\partial}W(K(a,D(a))).$$
By Lemma \ref{easylemma2}, $\bar{D}$ has a $B$-prolongation of the form
$$\bar{D}':K[a,D(a)]\to K[a,D(a),D(a),b'].$$
Using Lemma \ref{vergen}(1), we obtain that $\bar{D}'$ (hence also $\bar{D}$) has a $B$-regular realization on the field $K(a,D(a))$.
\\
$(2)\Rightarrow (3)$ and $(3)\Rightarrow (4)$. Obvious.
\\
$(4)\Rightarrow (1)$.
Let $\bar{D}'$ be a $B$-prolongation of $\bar{D}$. Then, we have:
$$(a,D(a),D(a),D'(D(a)))\in E(\Omega).$$
Hence we get:
$$(a,D(a))\in \im\left(\pi_E:E(L)\to W(L)\right).$$
Since $(a,D(a))$ is a generic point of $W$ over $K$, the morphism $\pi_E:E\to W$ is dominant.
\end{proof}
We formulate now our geometric axioms (see Remark \ref{partialv} for the definition of the map $\partial_V$).
\smallskip
\\
\textbf{Axioms for $B-\dcf$}
\\
The structure $(K,\partial)$ is a $B$-field such that for each pair $(V,W)$ of $K$-irreducible varieties, IF
\begin{itemize}
\item $W\subseteq \tau^{\partial}(V)$,

\item $W$ projects generically on $V$,

\item $E$ projects generically on $W$;
\end{itemize}
THEN there is $x\in V(K)$ such that $\partial_V(x)\in W(K)$.
\begin{remark}
\begin{enumerate}
\item The assumptions in the axioms above imply that $\tau^{\partial}(V)$ can not be the ``empty scheme'' (see Remark \ref{empty}).

\item It is standard to notice that the conditions above form a first-order scheme of axioms. It is explained in detail (for a very similar situation) e.g. in \cite[Remark 2.7(1)]{HK3}.
\end{enumerate}
\end{remark}

\begin{theorem}\label{mainthm}
Suppose that the structure $(K,\partial)$ is a $B$-field. Then the following are equivalent.
\begin{enumerate}
\item $(K,\partial)$ is existentially closed.

\item $(K,\partial)\models B-\dcf$.
\end{enumerate}
\end{theorem}
\begin{proof}
After all the preparatory results above, the proof closely follows the lines of the usual proofs in this context. More precisely, we follow the proofs of \cite[Theorem 2.1]{K2} and \cite[Theorem 2.17]{BK}.
\\
$(1)\Rightarrow (2)$. Let $(V,W)$ be a pair of varieties satisfying the assumptions of the axioms for the theory $B-\dcf$. Since the projection map $W\to V$ is dominant, there are $a,a'$ (tuples in $\Omega$) such that
$$V=\locus_K(a),\ \ \ W=\locus_K(a,a').$$
Since $W\subseteq \tau^{\partial}(V)$, by Lemma \ref{easylemma2} we can assume that $a'=D(a)$ for a $B$-kernel $\bar{D}:K[a]\to K[a,D(a)]$. By Prop. \ref{kerprol}, there is a $B$-regular realization $(L,D')$ of $\bar{D}$. Clearly, $a\in V(L)$ and $(a,D(a))\in W(L)$, so the $B$-field $(L,D')$ satisfies the conclusion of the axioms for $B-\dcf$. Since the $B$-field $(K,D)$ is existentially closed, it satisfies this conclusion as well.
\\
$(2)\Rightarrow (1)$. Let $\varphi(x)$ be a quantifier-free formula over $K$ in the language of $B$-fields. By the usual tricks, we can assume that the theory of $B$-fields implies the following:
$$\varphi(x)\ \ \ \Leftrightarrow\ \ \ \chi(x,\partial(x))$$
for a quantifier-free formula $\chi(x,y)$ in the language of fields. Assume that there is a $B$-field extension $(K,\partial)\subseteq (L,\partial')$ and a tuple $d$ from $L$ such that
$$(L,\partial')\models \chi(d,\partial'(d)).$$
We define:
$$V=\locus_K(d),\ \ \ W=\locus_K(d,\partial'(d)).$$
By Lemma \ref{easylemma2} and Prop. \ref{kerprol}, the pair $(V,W)$ satisfies the assumptions of axioms for $B-\dcf$. Hence there is $a\in V(K)$ such that $\partial_V(a)\in W(K)$. Therefore we have
$$(K,\partial)\models \exists x\ \chi(x,\partial(x)),$$
which is exactly what we wanted to show.
\end{proof}
In the next general result we do not assume that $B$ is local.
\begin{cor}\label{supermain}
The theory of fields with $B$-operators has a model companion if and only if the following two conditions are satisfied:
\begin{enumerate}
\item the nilradical of $B$ coincides with the kernel of the Frobenius homomorphisms on $B$;

\item $B$ is either local or
$$B\cong_{\ka} \ka_1\times \ldots \times \ka_n,$$
where each $\ka\subseteq \ka_i$ is a finite separable field extension.
\end{enumerate}
\end{cor}
\begin{proof}
Theorem \ref{nomc} says exactly that if the theory of fields with $B$-operators has a model companion, then the two conditions above are satisfied.

For the opposite implication, assume that these two conditions are satisfied. If $B$ is local, then we are done by Theorem \ref{mainthm}. Hence, we can assume that
$$B\cong_{\ka} \ka_1\times \ldots \times \ka_n,$$
where each $\ka\subseteq \ka_i$ is a finite separable field extension. Let $(K,\partial)$ be a $B$-field, and $K\subseteq K'$ be an algebraic field extension.  It is easy to see that there is a $B$-field extension $(K,\partial)\subseteq (K',\partial')$. Therefore, each existentially closed $B$-field contains $\ka^{\alg}$. Since each extension $\ka\subseteq \ka_i$ is finite and separable, we get that
$$B\otimes_{\ka}\ka^{\alg}\cong_{\ka^{\alg}}\left(\ka^{\alg}\right)^{\times e}.$$
Thus each existentially closed $B$-field $(K,\partial)$ is inter-definable with $K$ together with $d=e-1$ automorphisms (see Example \ref{ex}(1)). It is well-known that the theory of fields with $d$ endomorphisms has a model companion (see e.g. \cite{acfa1} for the case of $d=1$, and \cite[Section 2]{KiPi} for the case of $d>1$) called, in our positive characteristic case, ACFA$_{p,d}$. Hence the theory of $B$-fields has a model companion which is inter-definable with the theory ACFA$_{p,d}$ together with the axioms of $\ka$-algebras and $\ka$-algebra endomorphisms.
\end{proof}

\section{Model-theoretic properties of $\bdcf$}\label{secstable}
As in the previous section, we still use the fixed $\ka,B,\pi_B,e$ from the beginning of Section \ref{secoppro}. We assume that the theory $\bdcf$ exists, that is $B$ satisfies the conditions from items $(1)$ and $(2)$ in Corollary \ref{supermain}. We deal briefly with the non-local case (when the theory $\bdcf$ is inter-definable with the theory ACFA$_{p,d}$) in Section \ref{secacfapd}. Afterwards, we assume in this section that $B$ is local. To show that the theory $\bdcf$ is stable in this case, one should follow the proof of Theorem 2.4 from \cite{K2} and check whether all the corresponding differential algebraic facts also hold in the context of $B$-operators. We still assume that all the fields we consider are subfields of a big algebraically closed field $\Omega$.

\subsection{Linear disjointness}\label{secwro}
In this subsection, we show that for $B$ satisfying Assumption \ref{ass2}, strict $B$-fields are ``$B$-perfect'' (see Theorem \ref{strict}), which will be enough to get the amalgamation property and quantifier elimination in an appropriate language (see Section \ref{secqe}).

Let $\ka\subseteq K$ be a field extension and for any $K$-vector space $W$, we denote
$$W_B:=W\otimes_{\ka}B\cong_K W\otimes_{K}K_B$$
(an isomorphism of $K$-vector spaces). Assume now that $\partial:K\to K_B$ is a $B$-operator and $M$ is a $K_B$-module. We denote by $\partial^*(M)$ the $K$-vector space structure on $M$ which is the restriction of scalars along $\partial:K\to K_B$, i.e. for $\alpha\in K$ and $m\in\partial^*(M)$, we have:
$$\alpha\cdot m:=\partial(\alpha)m.$$
\begin{definition}
Let $W$ be a $K$-vector space and $V$ be a $K$-vector subspace of $W$.
\begin{enumerate}
\item A \emph{$\partial$-operator from $V$ to $W$} is a map $D:V\to W_B$, which is \emph{$\partial$-linear}, i.e. it is a $K$-linear map
$$D:V\to \partial^*(W_B).$$

\item For $D$ as above, we define the \emph{space of constants of $D$} as:
$$V^D:=\{v\in V\ |\ D(v)=v\otimes 1_B\}.$$
\end{enumerate}
\end{definition}
\begin{remark}
It is easy to see that $\partial:K\to K_B$ is a (unique) $\partial$-operator from $K$ to $K$. Then $K^{\partial}$ is a subfield of $K$, and for any $\partial$-operator $D:V\to W$, $V^D$ is a vector space over $K^{\partial}$.
\end{remark}
We need a lemma about exterior powers of $\partial$-operators.
\begin{lemma}\label{extpow}
Let $D:V\to W_B$ be a $\partial$-operator and $n\in \Nn$. Then there is a natural $\partial$-operator
$$\Lambda_D^{n}:\Lambda_K^{n}(V)\to \left(\Lambda_K^{n}(W)\right)_B$$
such that for any $v_1,\ldots,v_n\in V^D$, we have:
$$v_1\wedge \ldots \wedge v_n\in \left(\Lambda_K^{n}(V)\right)^{\Lambda_D^{n}}.$$
\end{lemma}
\begin{proof}
From the universal property of the exterior product (and the fact that the exterior product functor commutes with the extension of scalars), we get the following map:
$$\Lambda_K^{n}(W_B)\to \Lambda_{K_B}^{n}(W_B)\cong \left(\Lambda_K^{n}(W)\right)_B.$$
Composing the map above with the following exterior power map:
$$\Lambda_K^{n}(D):\Lambda_K^{n}(V)\to \Lambda_K^{n}(W_B),$$
we get our $\partial$-operator $\Lambda_D^{n}:\Lambda_K^{n}(V)\to \left(\Lambda_K^{n}(W)\right)_B$.
\end{proof}
We denote $K^{\partial}$ by $C$ and prove our linear disjointness theorem.
\begin{theorem}\label{lidi}
Let $D:V\to W_B$ be a $\partial$-operator. Then the natural map
$$K\otimes_CV^D\to V$$
(coming from the scalar multiplication) is injective.
\end{theorem}
\begin{proof}
Let us assume first that $\dim_K(V)=1$. If the result does not hold, then there are $u,v\in V^D$, which are $C$-linearly
independent. Since $\dim_K(V)=1$, there is invertible $\alpha\in K$ such that $u=\alpha v\in V^D\setminus \{0\}$. Applying
$D$, we get:
$$\alpha v=D(\alpha v)=\partial(\alpha)D(v)=\partial(\alpha)v,$$
so $\alpha=\partial(\alpha)$, a contradiction.

Let us take now $V$ arbitrary, and a minimal counterexample to our result, i.e. for some minimal $n>0$, there are $v_0,\dots,v_n\in V^D$, which are linearly independent over $C$ but linearly dependent over $K$. Since we have proved the result for $\dim_K(V)=1$, we get that $n>1$. Without loss of generality, we can assume that $V$ coincides with the $K$-linear span of $v_0,\dots,v_n$. By our minimality assumption, $\dim_K(V)=n$. By Lemma \ref{extpow}, we get an appropriate $\partial$-operator
$$\Lambda_D^{n}:\Lambda_K^{n}(V)\to \left(\Lambda_K^{n}(W)\right)_B.$$
Then we have $\dim_K(\Lambda_K^{n}(V))=1$ and
$$a:=v_0\wedge v_1\wedge \ldots \wedge v_{n-2} \wedge v_{n-1}\in \Lambda_K^{n}(V)^{\Lambda_D^{n}},\ \ b:=v_0\wedge v_1\wedge \ldots \wedge v_{n-2} \wedge v_{n}\in \Lambda_K^{n}(V)^{\Lambda_D^{n}}.$$
By the truth of the result in the one-dimensional case, there is (without loss of generality) a non-zero $\alpha\in C$ such that $b=\alpha a$. Hence we get that:
$$v_0\wedge v_1\wedge \ldots \wedge v_{n-2}\wedge (v_{n}-\alpha v_{n-1})=0.$$
Therefore we obtain the folowing:
\begin{itemize}
\item $v_0,v_1,\ldots, v_{n-2}, v_{n}-\alpha v_{n-1}$ are $K$-linearly dependent;

\item $v_0,v_1,\ldots, v_{n-2}, v_{n}-\alpha v_{n-1}\in V^D.$
\end{itemize}
By the minimality of $n$, the above two items imply that $v_0,v_1,\ldots, v_{n-2}, v_{n}-\alpha v_{n-1}$ are $K$-linearly dependent. Therefore, $v_0,\dots,v_n$ are linearly dependent over $C$, which is a contradiction.
\end{proof}
\begin{cor}\label{lindis}
Suppose that $(K,\partial)\subseteq (L,\partial')$ is an extension of $B$-fields. Then $K$ is linearly disjoint from $L^{\partial'}$ over $K^{\partial}$ (inside $L$).
\end{cor}
\begin{proof}
It follows directly from Theorem \ref{lidi} by taking $V=W=L$ and $D=\partial'$.
\end{proof}
\begin{remark}
We comment here on Theorem \ref{lidi} and Corollary \ref{lindis}.
\begin{enumerate}
  \item We discuss several special cases.
\begin{enumerate}
  \item If $B=\ka[X]/(X^2)$, then we get the well-known result (see e.g. \cite[Theorem 3.7]{kapint} and the paragraph after its proof) saying that if $K\subseteq M$ is a differential field extension, then the constants of $M$ are linearly disjoint from $K$ over the constants of $K$. It is classically being proved using the Wro\'{n}skian method. Our proof avoids almost any computations, but we can see some kind of a ``shade'' of the Wro\'{n}skian determinant there, since inside the $n$-th exterior power of an $n$-th dimensional space, the elementary wedge products are exactly the determinants of the corresponding matrices (after choosing a basis).

      We also get similar independence results for $B=\ka[X]/(X^e)$ for any $e>1$, i.e. for \emph{higher iterative derivations} also known as  (non-iterative!) \emph{Hasse-Schmidt derivations}. We do not know if such independence results were already known.
  \item If $B=\ka\times \ka$, then $B$-operators are the same as endomorphisms, and the constants in our sense coincide with the constants of endomorphisms. The corresponding linear independence result is probably well-known, and it is also very easy to show.

  \item If we take fiber products over $\ka$ of the $\ka$-algebras $B$ from the two items above, then the corresponding operators are finite sequences of endomorphisms and (higher) derivations. As far as we know, the corresponding linear disjointness result is known only in the case of several \emph{commuting} derivations, and it is usually proved using generalized Wro\'{n}skians.
\end{enumerate}

  \item Our Theorem \ref{lidi} is surprisingly general. Actually, the $\ka$-algebra $B$ there is totally arbitrary (e.g. it need not be finite dimensional), we do not need the $\ka$-algebra map $\pi_B:B\to \ka$ and the characteristic of $\ka$ is arbitrary.
\end{enumerate}
\end{remark}
Using the field of constants $K^{\partial}$, Lemma \ref{frl} can be translated as follows.
\begin{lemma}\label{frl2}
Suppose $B$ satisfies Assumption \ref{ass2} and $(K,\partial)$ is a field with $B$-operator. Then we have $K^p\subseteq K^{\partial}$.
\end{lemma}
We need one more definition generalizing the classical one (from the differential case).
\begin{definition}
Suppose that $B$ satisfies Assumption \ref{ass2} and $(K,\partial)$ is a $B$-field. Then $(K,\partial)$ is \emph{strict}, if $K^{\partial}=K^p$.
\end{definition}
We can prove now the main result of this subsection.
\begin{theorem}\label{strict}
Suppose that $B$ satisfies Assumption \ref{ass2}. Then strict $B$-fields are ``$B$-perfect'', i.e. if $(K,\partial)\subseteq (L,\partial')$ is a $B$-field extension and the $B$-field $(K,\partial)$ is strict, then the field extension $K\subseteq L$ is separable.
\end{theorem}
\begin{proof}
By Corollary \ref{lindis}, $K$ is linearly disjoint from $L^{\partial'}$ over $K^{\partial}=K^p$. Since $L^p\subseteq L^{\partial'}$ (by Lemma \ref{frl2}), we get that $K$ is linearly disjoint from $L^p$ over $K^p$, i.e. the field extension $K\subseteq L$ is separable.
\end{proof}

\subsection{Several endomorphisms}\label{secacfapd}
We assume in this subsection that $B$ satisfies the conditions from items $(1)$ and $(2)$ in Corollary \ref{supermain} and that $B$ is not local. By the proof of Corollary \ref{supermain}, the theory $\bdcf$ is inter-definable with the theory ACFA$_{p,d}$ (the model companion of the theory of fields of characteristic $p$ with $d$ endomorphisms) together with the axioms of $\ka$-algebras and $\ka$-algebra endomorphisms. It is possible that this theory has been already described in the literature, but we could not find any reference, hence we will provide short arguments below.

Let $\mathcal{L}$ be the language of fields expanded by the constants given by elements of $\ka$, $T=\mathrm{Th}(\mathrm{Fields})\cup \mathrm{Diag}(\ka)$ (the models of $T$ are exactly the field extensions of $\ka$), and $G=F_d$ be the free group on $d$ (free) generators. Then we are in the set-up of \cite{Hoff3}, and we know (by Corollary \ref{supermain}) that the theory $T^{\mathrm{mc}}_G$ (the model companion of the theory of models of $T$ with actions of $G$ by $\mathcal{L}$-automorphisms) exists and coincides with the theory $\bdcf$. Therefore, \cite[Theorem 4.22]{Hoff3} (see also \cite[Remark 4.20(1)]{Hoff3}) implies that the theory $\bdcf$ is simple and the forking relation there is the obvious one, i.e. $a$ is independent from $b$ over $c$ if and only if $G\cdot a$ (the orbit of $a$ under the action of $G=F_k$) is $\mathrm{ACF}_p$-independent from $G\cdot b$ over $G\cdot c$. Then, the standard arguments give the elimination of imaginaries for $\bdcf$, for example one can repeat the argument from the proof of \cite[Theorem 5.12]{MS2}.

\subsection{Amalgamation Property and Quantifier Elimination}\label{secqe}
From now on till the end of Section \ref{secstable}, we will be working under Assumption \ref{ass2}. To show the amalgamation property for the theory $\bdcf$, we can follow the lines of the proof of \cite[Fact 1.10]{K2}.

We introduce here several languages which we will use:
\begin{itemize}
  \item $L$ is the language of rings;
  \item $L^B$ is the language of rings with $d=e-1$ extra unary function symbols and extra constant symbols for the elements of $\ka$ (the language used in Section \ref{secaxioms});
\item $L_{\lambda_0}$ is the language of rings with an extra unary function symbol $\lambda_0$ (for the $p$-th root function);
  \item $L_{\lambda_0}^B:=L_{\lambda_0}\cup L^B$.
\end{itemize}
Note that any field $K$ of characteristic $p$ naturally becomes an $L_{\lambda_0}$-structure, where $\lambda_0$ is understood as:
$$\lambda_0:K\to K,\ \ \ \lambda_0(x)= \begin{cases} x^{1/p}\text{ for }x\in K^p, \\ 0\ \ \ \ \text{ for }x\notin K^p. \end{cases}$$
We will often use the obvious result saying that extensions of fields $K\subseteq M$ of characteristic $p$ preserving the $\lambda_0$-function (i.e. the  $L_{\lambda_0}$-extensions) are exactly the field extensions such that $M^p\cap K=K^p$ (see e.g. \cite[Lemma 4.1]{HK}).
\begin{prop}\label{ap}
The theory $\bdcf$ has Amalgamation Property in the language $L_{\lambda_0}^B$.
\end{prop}
\begin{proof}
We follow closely the proof from \cite{K2} (which was in turn based on a proof from \cite{Zieg2}). Let us take $\mathbf{K}_1=(K,\partial^1),\mathbf{K}_2=(K,\partial^2)$ which are models of $\bdcf$ and a common $L^B_{\lambda_0}$-substructure
$$\mathbf{K}=(K,\partial)\subseteq \mathbf{K}_1,\mathbf{K}_2.$$
Then $K$ is a domain and using Lemma \ref{vergen}(1), we can assume that $K$ is a field. By Lemma \ref{gencon}(2), the $B$-fields $\mathbf{K}_1,\mathbf{K}_2$ are strict (being existentially closed). Hence it is easy to see that the $B$-field $\mathbf{K}$ is strict: if $a\in \mathbf{K}^{\partial}$, then $a\in \mathbf{K}_1^{\partial^1}=K_1^p$, so $a\in K^p$, since $K\subseteq K_1$ is an $L_{\lambda_0}$-extension. By Theorem \ref{strict}, the field extensions $K\subseteq K_1,K_2$ are separable.

We can assume that $K_1$ is algebraically disjoint from $K_2$ over $K$ (inside $\Omega$). Using Lemma \ref{vergen}(2) (since separable algebraic field extensions are \'{e}tale), we can assume that $K$ is separably closed (by replacing all the fields $K,K_1,K_2$ with their separable closures in $\Omega$). Then the extensions $K\subseteq K_1,K\subseteq K_2$ are regular (see \cite[Lemma 2.6.4]{FrJa}). By \cite[Lemma 2.6.7]{FrJa}, $K_1$ is linearly disjoint from $K_2$ over $K$. Therefore the tensor product $K_1\otimes_KK_2$ is a domain. By Prop. \ref{tensor}, there is a $B$-operator on $K_1\otimes_KK_2$ extending the $B$-operators on $K_1$ and $K_2$. By Lemma \ref{vergen}(1), the $B$-operator on $K_1\otimes_KK_2$ uniquely extends to the field of fractions, which gives the amalgamation of $\mathbf{K}_1$ and $\mathbf{K}_2$ over $\mathbf{K}$ we were looking for.
\end{proof}
We can conclude now our quantifier elimination result.
\begin{theorem}\label{qe}
The theory $\bdcf$ has Quantifier Elimination in the language $L_{\lambda_0}^B$.
\end{theorem}
\begin{proof}
It is enough to notice that:
\begin{itemize}
\item since the theory $\bdcf$ is a model companion of the theory of fields with $B$-operators (Theorem \ref{mainthm}), $\bdcf$ is model complete;

\item the theory $\bdcf$ has the amalgamation property in the language $L^B_{\lambda_0}$ (Prop. \ref{ap}).
\end{itemize}
Since any model complete theory with the amalgamation property admits quantifier elimination, the result is proved.
\end{proof}

\subsection{Stability}

We fix a monster model $(\mathfrak{C},\partial)\models\bdcf$.
\begin{remark}
The field $\mathfrak{C}$ is separably closed and not perfect. More precisely, $\mathfrak{C}$ has an infinite imperfection degree.
\end{remark}
\begin{proof}
The fact that $\mathfrak{C}$ is separably closed follows from Lemma \ref{vergen}(2) (separable algebraic field extensions are \'{e}tale). To show that $\mathfrak{C}$ has an infinite imperfection degree, it is enough to show (using Proposition \ref{tensor}, Lemma \ref{vergen}(1), Lemma \ref{gencon}(2) and Corollary \ref{lindis}) that for each $n>0$, there is a $B$-field structure $\partial$ on the field of rational functions $\ka(X_1,\ldots,X_n)$ such that:
\begin{equation}
\left[\ka(X_1,\ldots,X_n):\ka(X_1,\ldots,X_n)^{\partial}\right]\geqslant n.\tag{$*$}
\end{equation}
It is classical for $B=\ka[X]/(X^2)$, i.e. the case of derivations: one takes a derivation $\partial$ such that $\partial(X_i)=X_{i+1}$ for $i<n$ and $\partial(X_n)=0$, and uses the Wro\'{n}skian method (see e.g. the proof of \cite[Theorem 2.5]{Wo1}). In the general case, there is a $\ka$-algebra epimorphism $B\to \ka[X]/(X^2)$ (obtained by dividing first by a power of the maximal ideal of $B$, and then killing enough generators of the maximal ideal of this quotient). Using this epimorphism, we can lift the above derivation on $\ka(X_1,\ldots,X_n)$ to a $B$-operator satisfying $(*)$ (this lifting is possible by the universal property of $K$-algebras of polynomials and Lemma \ref{vergen}(1)).
\end{proof}
We denote the forking independence in $\mathfrak{C}$, considered as a separably closed field, by $\ind^{\scf}$. Similarly for types, algebraic closure, definable closure and groups of automorphisms, e.g. we use the notation $\acl^{\scf}$. On the other hand, $\acl^{\bdcf}$ corresponds to the algebraic closure computed in the $B$-field $(\mathfrak{C},\partial)$.

The following result is about pure fields. It comes from \cite{Wo2} and it also appears as \cite[Fact 2.3(i)]{K2}.
\begin{lemma}\label{fact23}
Consider subfields $K$, $M$ and $M'$ of $\mathfrak{C}$ such that the extensions $K\subseteq M,M\subseteq \mathfrak{C}$ and $K\subseteq M',M'\subseteq \mathfrak{C}$ are separable.
If $M$ is $p$-disjoint from $M'$ over $K$ in $\mathfrak{C}$, then the extension $MM'\subseteq \mathfrak{C}$ is separable.
\end{lemma}

\begin{lemma}\label{acl}
For any small subset $A$ of $\mathfrak{C}$, it follows
$$\acl^{\bdcf}(A)=\acl^{\scf}(\langle A\rangle_B),$$
where $\langle A\rangle_B$ denotes the $B$-subfield of $\mathfrak{C}$ generated by $A$.
\end{lemma}
\begin{proof}
The proof is standard, e.g. the proof of \cite[Lemma 4.10]{Hoff3} is a proof based on the same idea. We need to show that $E:=\acl^{\scf}(\langle A\rangle_B)$ is $B$-algebraically closed.
Assume not, and take $d\in\acl^{\bdcf}(E)\setminus E$. Let $(K,\partial)\preccurlyeq(\mathfrak{C},\partial)$ be such that $E\subseteq K$ and such that $K$ also contains the (finite) orbit of $d$ under the action of $\aut^{\bdcf}(\mathfrak{C}/E)$.

There is $f\in\aut^{\scf}(\mathfrak{C}/E)$ such that $f(K)$ is algebraically disjoint from $K$ over $E$. If $f(d)\in K$, then $d\in E$ (since $E$ is SCF-algebraically closed). Therefore $f(d)\not\in K$.

Let $\partial^f$ denote the $B$-operator on $f(K)$ such that for all $i\in \{0,1,\ldots,d\}$, we have:
$$\left(\partial^f\right)_i=f\circ \partial_i \circ f^{-1}.$$
By Prop. \ref{tensor}, there is a $B$-operator on $f(K)\otimes_E K$ extending $\partial$ on $K$ and $\partial^f$ on $f(K)$.
Similarly as in the proof of Prop. \ref{ap}, the field extension $E\subseteq K$ is regular, so $f(K)$ is linearly disjoint from $K$ over $E$ and  $f(K)\otimes_E K$ is a domain. By Lemma \ref{vergen}(1), the $B$-operator on $f(K)\otimes_E K$ extends uniquely to the field of fractions of $f(K)\otimes_E K$. This field of fractions above, can be embedded (as a $B$-field) over $K$ into $\mathfrak{C}$. Hence, we can assume that $f(K)$ is an elementary substructure of $(\mathfrak{C},\partial)$ and $f$ is a $B$-isomorphism. Therefore, $f$ extends to an element of $\aut^{\bdcf}(\mathfrak{C}/E)$. But then we get
$$f(d)\in \left(\aut^{\bdcf}(\mathfrak{C}/E)\cdot d\right)\setminus K,$$
which is a contradiction.
\end{proof}

We proceed now towards a description of the forking independence in the theory $\bdcf$.
We start with the following definition of a ternary relation on small subsets $A$, $B$, $C$ of the monster model $(\mathfrak{C},\partial)$:
$$A\ind^{\bdcf}_C B\qquad\iff\qquad\acl^{\bdcf}(A)\ind^{\scf}_{\acl^{\bdcf}(C)}\acl^{\bdcf}(B).$$
We will show that the ternary relation $\ind^{\bdcf}$ defined above satisfies the following properties of forking in stable theories.
\begin{enumerate}
\item[(P1)] (invariance) The relation $\ind^{\bdcf}$ is invariant under $\aut^{\bdcf}(\mathfrak{C})$.

\item[(P2)] (symmetry) For every small $A,B,C\subset\mathfrak{C}$, it follows that
$$A\ind^{\bdcf}_C B\qquad\iff\qquad B\ind^{\bdcf}_C A.$$

\item[(P3)] (monotonicity and transitivity) For all small $A\subseteq B\subseteq C\subset\mathfrak{C}$ and small $D\subset\mathfrak{C}$, it follows that
$$D\ind^{\bdcf}_A C\quad\iff\quad D\ind^{\bdcf}_A B\;\;\text{ and }\;\;D\ind^{\bdcf}_B C.$$

\item[(P4)] (existence) For every finite $a\subset\mathfrak{C}$ and every small $A\subseteq B\subset\mathfrak{C}$, there exists $f\in\aut^{\bdcf}(\mathfrak{C})$ such that $f(a)\ind^{\bdcf}_A B$.

\item[(P5)] (local character) For every finite $a\subset\mathfrak{C}$ and every small $B\subset\mathfrak{C}$, there is $B_0\subseteq B$ such that $|B_0|\leqslant\omega$ and $a\ind^{\bdcf}_{B_0}B$.

\item[(P6)] (finite character) For every small $A,B,C\subset\mathfrak{C}$, we have:
\\
 $A\ind^{\bdcf}_C B$ if and only if $a\ind^{\bdcf}_C B$ for every finite $a\subseteq A$.

\item[(P7)] (uniqueness over a model) Any complete type over a model is stationary.
\end{enumerate}
The properties (P1), (P2), (P3) and (P6) follow easily from the definition of $\ind^{\bdcf}$. The property (P5) follows from the local character and the finite character of $\ind^{\scf}$.
\begin{lemma}\label{lemma.P4}
The property (P4) (existence) holds.
\end{lemma}

\begin{proof}
Let $a\subseteq\mathfrak{C}$ be finite and let $A\subseteq B$ be small subsets of $\mathfrak{C}$. Consider small $(K,\partial)\preccurlyeq(\mathfrak{C,\partial})$ containing $a$, $B$, and $E=\acl^{\bdcf}(A)$. There exists $f\in\aut^{\scf}(\mathfrak{C})$ such that $f(K)$ is algebraically disjoint from $K$ over $E$. Exactly as in the proof of Lemma \ref{acl}, we see that $f(K)$ is an elementary $B$-substructure of $\mathfrak{C}$ and that $f:K\to f(K)$ extends to an element of $\mathrm{Aut}^{\bdcf}(\mathfrak{C})$. Then we get (by the definition of $\ind^{\bdcf}$ and by Lemma \ref{acl}) that
$$f(a)\ind^{\bdcf}_A B,$$
which was our goal.
\end{proof}

\begin{lemma}\label{lemma.P7}
The property (P7) (uniqueness over a model) holds.
\end{lemma}

\begin{proof}
Let us take a $B$-elementary subfield $K$ of $\mathfrak{C}$, an elementary extension of $B$-fields $K\preccurlyeq M$ (inside $\mathfrak{C}$) and $a,b\in \mathfrak{C}$. We assume that:
$$\tp^B(a/K)=\tp^B(b/K),\ \ \ \ \ \  a\ind^{\bdcf}_K M ,\ \ \ \ \ \  b\ind^{\bdcf}_K M.$$
Our aim is to show that $\tp^B(a/M)=\tp^B(b/M)$.

Suppose that $f\in\aut^{\bdcf}(\mathfrak{C}/K)$ is such that $f(a)=b$. We consider
$$K_a:=\dcl^{\bdcf}(Ka),\ \ \ \ \ K_b:=\dcl^{\bdcf}(Kb)$$
and notice that $f$ is a $B$-isomorphism between $(K_a,\partial)$ and $(K_b,\partial)$.

Note that $M=\dcl^{\bdcf}(M)$. The extension $K\subseteq M$ is regular, so as in the proof of Lemma \ref{acl} we get that
$K_a\otimes_K M$ and $K_b\otimes_K M$ are domains and the map
$$f|_{K_a}\otimes\id_M:K_a\otimes_K M\to K_b\otimes_K M$$
is a $B$-isomorphism, which extends (by Lemma \ref{vergen}(1)) to the isomorphism $\widetilde{f}$ of the fields of fractions. Again, as in the proof of Lemma \ref{acl}, the regularity of $K\subseteq M$ and the algebraic disjointness of $K_a$ with $M$ over $K$ lead to linear disjointness of $K_a$ and $M$ over $K$, hence there exists a $B$-isomorphism $f_a$ from $K_aM$ to the field of fractions of $K_a\otimes_K M$ taking $a$ to $a\otimes 1$.
Similarly, there exists a $B$-isomorphism $f_b$ from $K_bM$ to the field of fractions of $K_b\otimes_K M$ taking $b$ to $b\otimes 1$.
Composing the $B$-isomorphisms $f_a$, $\widetilde{f}$ and $f_b$ gives us a $B$-isomorphism $h:K_aM\to K_bM$ such that $h(a)=b$.

Since $K, K_a, K_b$ and $M$ are definably closed subsets of $\mathfrak{C}$, by Theorem \ref{strict} the extensions $K\subseteq K_a\subseteq \mathfrak{C}$, $K\subseteq K_b\subseteq \mathfrak{C}$ and $K\subseteq M\subseteq \mathfrak{C}$ are separable. By Lemma \ref{fact23}, the extensions $K_aM\subseteq\mathfrak{C}$ and $K_bM\subseteq\mathfrak{C}$ are separable (linear disjointness implies $p$-independence, so Lemma \ref{fact23} can be applied). In particular, $K_aM$ and $K_bM$ are $L_{\lambda_0}^B$-substructures of $\mathfrak{C}$. By Theorem \ref{qe}, there exists $\hat{h}\in\aut^{\bdcf}(\mathfrak{C}/K)$ such that $\hat{h}|_{K_aM}=h$, hence $\tp^B(a/M)=\tp^B(b/M)$.
\end{proof}

\begin{theorem}\label{bdcf.stable.thm}
The theory $\bdcf$ is stable, not superstable and the relation $\ind^{\bdcf}$ coincides with the forking independence.
\end{theorem}

\begin{proof}
The underlying fields of models of $\bdcf$ are separably closed and not perfect, hence $\bdcf$ is not superstable. Stability and the description of forking independence follows from \cite[Fact 2.1.4]{Kimsim}, since the relation $\ind^{\bdcf}$ satisfies the properties (P1)--(P7).
\end{proof}

\section{Generalizations and further directions}\label{secend}
In this section, we discuss some other topics related with model theory of fields with free operators.

\subsection{No jet spaces methods}
It does not look possible at this moment to have here a positive characteristic version of the jet spaces techniques and results from \cite{MS2}. The reason is rather simple: having them would imply (in particular) Zilber's trichotomy for any theory of the shape $\bdcf$ (for $B$ satisfying Assumption \ref{ass2}), and this trichotomy is unknown even in the ``simplest'' case of DCF$_p$, that is for the case of $B=\ka[X]/(X^2)$.

\subsection{Elimination of Imaginaries}
Similarly as in the subsection above, if $B$ satisfies Assumption \ref{ass2} then we can not hope for the elimination of imaginaries for the theories of the shape $\bdcf$ in any language, which we have considered in this paper. The reason is again the same: as noted in \cite[Remark 4.3]{MeWo}, the ``simplest'' theory DCF$_p$ has no elimination of imaginaries (in any of these languages).

However, in the case of a non-local $B$, we do get the elimination of imaginaries for the theory $\bdcf$, as was explained in Section \ref{secacfapd}.

\subsection{Derivations of Frobenius}
This paper does \emph{not} generalize the results from \cite{K2}, since (rather surprisingly) derivations of Frobenius do not fit to the set-up of \cite{MS2}, which we explain briefly below. Derivations of Frobenius are also controlled by a representable functor and natural trasformations
$$\mathcal{B}:\mathrm{Alg}_{\ka}\to \mathrm{Alg}_{\ka},\ \ \ \iota_{\mathcal{B}}:\id\to \mathcal{B},\ \ \ \pi_{\mathcal{B}}:\mathcal{B}\to \id$$
($\mathcal{B}(R)$ is denoted by $R_{(1)}$ in \cite{K2}), but this functor is not of the form $\cdot\otimes_{\ka}B$ for any $\ka$-algebra $B$, since there are field extensions $\ka\subseteq K,\ka\subseteq L$ such $\dim_K(\mathcal{B}(K))$ is finite and $\dim_L(\mathcal{B}(L))$ is infinite. It looks like one can still develop a theory of $\mathcal{B}$-operators for a more general class of functors $\mathcal{B}$ than the ones considered in \cite{MS2} (which are the functors of the form $\cdot\otimes_{\ka}B$) in such a way that this new theory will cover derivations of Frobenius as well; and then (using the set-up provided by this new theory) one could possibly prove results generalizing both the results of this paper and the results of \cite{K2}, but this will be done elsewhere.

\subsection{Formal group actions}
A common context between the results of this paper and the results from \cite{BK} may be found in the set-up of \emph{formal group actions} on fields. This common context was actually the research topic of another working group at the July 2016 \c{S}irince workshop, which consisted of the third author, the fourth author, Rahim Moosa and Thomas Scanlon. In \cite{BK}, the conjectural condition on a group $G$ being equivalent to companionability of the theory of fields with $G$-actions is \emph{virtual freeness}. In this paper, the condition giving companionability is Assumption \ref{ass2} from Section \ref{secoppro}.

We would like to emphasize that the geometric axioms from this paper and the geometric axioms from \cite{BK} (and from many other places) have the same form: we have a subvariety $W$ of the prolongation $\tau^{\partial}(V)$ of a variety $V$ and we are looking for a rational point of $V$ whose natural image by $\partial$ (the operator we consider) in $\tau^{\partial}(V)$  belongs actually to $W$.

\subsection{More geometric approach}
The formalism of $B$-operators can be viewed more geometrically by
considering $B$ as the algebra of function of the corresponding affine
scheme $M=\spec(B)$. The map $\pi_B:B\to \ka$ then corresponds to a
base-point $*\in M$. A $B$-operator on $R$ is an ``action'' of $M$ on
$X=\spec(R)$, i.e. it is a map
$$a:M\times_{\ka} X\to X,$$
such that the
restriction of $a$ to $*\times_{\ka} X=X$ is the identity. For example, in
the case of an endomorphism, $M=\{*,\sigma\}$ is a two-point set. The
prolongation functor then assigns, to a scheme $X$, the scheme
$\tau^\partial(X)=X^M$ of functions from $M$ to $X$, and the map from it
to $X$ is given by evaluation at $*$. The reader is advised to consult \cite[Section 3]{Kam} for
more details on this point of view.
\bibliographystyle{plain}
\bibliography{harvard}

\end{document}